\documentclass[12pt,a4paper,reqno]{article}

\usepackage{amsmath}
\usepackage{amsthm}
\usepackage{enumerate}
\usepackage{amssymb}
\usepackage{mathrsfs}

\usepackage{verbatim}
\usepackage{url}
\usepackage[latin1]{inputenc}

\usepackage{caption}
\usepackage{graphicx,color}
\usepackage{tikz}
\usetikzlibrary{hobby}

\newcommand\ett{\mathbf1}

\def\Z{{\mathbb Z}}

\def\R{{\mathbb R}}

\def\Pr{{\mathbb P}}

\def\org{\mathbf{0}}

\def\nbf{\mathbf{n}}
\def\mbf{\mathbf{m}}

\def\0{\mathbf{0}}

\def\ett{\mathbf{1}}

\def\be{\begin{equation}}
\def\ee{\end{equation}}
\def\bea{\begin{equation*}}
\def\eea{\end{equation*}}
\def\bal{\begin{aligned}}
\def\eal{\end{aligned}}

\def\eps{\varepsilon}

\DeclareMathOperator{\E}{{\mathbb E}}

\def\le{\leqslant}
\def\leq{\leqslant}
\def\ge{\geqslant}
\def\geq{\geqslant}

\newtheorem{thm}{Theorem}
\newtheorem{lma}[thm]{Lemma}

\theoremstyle{remark}

\newtheorem{preex}[thm]{Example}

\theoremstyle{definition}

\newcommand{\sss}{\scriptscriptstyle}
\newcommand{\Tk}{T^{\sss (k)}}
\newcommand{\muk}{\mu^{\sss (k)}}
\newcommand{\Hl}{\mathcal{L}_\eps}
\newcommand{\Hr}{\mathcal{R}_\eps}
\def\coex{\mathscr{C}}

\begin{document}

\title{The two-type Richardson model in the half-plane}
\author{Daniel Ahlberg\thanks{Department of Mathematics, Stockholm University; {\tt daniel.ahlberg@math.su.se}} \and Maria Deijfen\thanks{Department of Mathematics, Stockholm University; {\tt mia@math.su.se}} \and  Christopher Hoffman \thanks{Department of Mathematics, University of Washington; {\tt hoffman@math.washington.edu}} }
\date{August 2018}

\maketitle

\begin{abstract}
\noindent The two-type Richardson model describes the growth of two competing infection types on the two or higher dimensional integer lattice. For types that spread with the same intensity, it is known that there is a positive probability for infinite coexistence, while for types with different intensities, it is conjectured that infinite coexistence is not possible. In this paper we study the two-type Richardson model in the upper half-plane $\Z\times\Z_+$, and prove that coexistence of two types starting on the horizontal axis has positive probability if and only if the types have the same intensity.
\end{abstract}

\section{Introduction}

In 1998, H\"aggstr\"om and Pemantle~\cite{HP_twotype} introduced a model for competing growth on $\Z^2$ known as the two-type Richardson model. Two competing entities, here referred to as type 1 and type 2 infection, initially occupy one site each of the $\Z^2$ nearest-neighbor lattice. As time evolves each uninfected site is occupied by type $i$ at rate $\lambda_i$ times the number of type $i$ neighbors. An infected site remains in its state forever,
implying that the model indeed defines a competition scheme between the types.

Regardless of the values of the intensities, both types clearly have a positive probability of winning by surrounding the other type at an early stage. Attention hence focuses on the event $\coex$ that both types simultaneously grow to occupy infinitely many sites; this is referred to as \emph{coexistence} of the two types.  Deciding whether or not $\coex$ has positive probability is non-trivial since it cannot be achieved on any finite part of the lattice. By time-scaling and symmetry we may restrict to the case $\lambda_1=1$ and $\lambda_2=\lambda>1$. The conjecture, due to H\"aggstr\"om and Pemantle~\cite{HP_twotype}, then is that $\coex$ has positive probability if and only if $\lambda=1$. The \emph{if}-direction of the conjecture was proved in~\cite{HP_twotype}, and extended to higher dimensions independently by Garet and Marchand~\cite{GM_coex} and Hoffman~\cite{Hoff_coex}, using different methods. As for the \emph{only if}-direction, H\"aggstr\"om and Pemantle \cite{HP_absence} showed in 2000 that coexistence is possible for at most countably many values of $\lambda$. Ruling out coexistence for \emph{all} $\lambda>1$ remains a seemingly challenging open problem.

In this paper we study the analogous problem in the upper half-plane $\Z\times\Z_+=\{(x,y):y\ge0\}$ with $(0,0)$ initially occupied by type 1 and $(1,0)$ initially occupied by type 2, and show that coexistence has positive probability if and only if $\lambda=1$. That coexistence is possible for $\lambda=1$ follows from similar arguments as in the full plane, so the novelty lies in proving the \emph{only if}-direction.

\begin{thm}\label{thm:main}
Consider the two-type Richardson model on $\Z\times\Z_+$ with $(0,0)$ and $(1,0)$ initially of type 1 and 2, respectively. Then we have that $\Pr(\coex)>0$ if and only if $\lambda=1$.
\end{thm}

Some readers might suspect that the arguments used to prove this result could be adaptable to settle the H\"aggstr\"om-Pemantle conjecture in the full plane. This however is most likely not the case. It is known that, on the event of coexistence in the full plane, the speed of the growth is determined by the weaker type; see e.g.\ \cite[Proposition 2.2]{HP_absence}. This means that, in order not to grow too fast, the stronger type must survive by maintaining a meandering path surrounded by the weaker type. In fact, it can be shown that the fraction of the infected sites occupied by the stronger type is vanishing; see \cite{GM_invisible}. The crucial point in our half-plane argument is that infinite survival for the stronger type implies that it must occupy all sites along the positive horizontal axis. We use this to show that it will thereby grow fast enough to eventually surround the weaker type. Note that the role of the initial configuration is important for this argument. We have not been able to adapt the argument to rule out coexistence in the half-plane when the initial position of the stronger type is not connected to the horizontal axis. Indeed, working with general initial configurations seems to make the problem as hard as in the full plane. We remark that, in the full plane, it is shown in \cite{initial} that the initial configuration is irrelevant for the possibility of infinite coexistence, but that argument does not apply here.

One way of constructing the two-type process is by independently assigning a unit exponential random weight $\tau(e)$ to each nearest-neighbor edge $e$ of the lattice. The time required for type 1 to traverse an edge $e$ is then given by the associated weight $\tau(e)$, and the time for type 2 is $\lambda^{-1}\tau(e)$. Indeed, this construction provides a coupling of the two-type models for all $\lambda\ge1$ simultaneously. The curious partial result of~\cite{HP_absence} is derived based on this coupling by showing that, in the probability measure underlying the coupling, there is almost surely at most one value of $\lambda$ for which coexistence may occur. That coexistence occurs with positive probability for at most countably many $\lambda\ge1$ is an easy consequence of this.

There are a number of proofs of coexistence for the case when the types have the same intensities, and (at least some of) these arguments can be adapted to prove the \emph{if}-direction of Theorem \ref{thm:main}. We shall however offer an alternative proof, since it is a simple by-product of the arguments required to prove the \emph{only if}-direction of the theorem. To rule out coexistence for $\lambda>1$, we shall develop an argument inspired by the work of Blair-Stahn~\cite{BS_thesis}, and that incorporates elements of Busemann functions introduced by Hoffman~\cite{Hoff_coex,Hoff_geodesics}. Nevertheless, the proof will be a self-contained and elementary deduction from standard results in first-passage percolation.

The two-type Richardson model can be viewed as a two-type version of first passage percolation with exponential edge weights. One of the most fundamental results for first passage percolation is the shape theorem, asserting that the infected set at time $t$ converges on the scale $t^{-1}$ to a deterministic convex set $A$. In order to describe the structure of the proof of Theorem \ref{thm:main}, let $\theta$ denote the maximal angle between any supporting line of $A$ in the first coordinate direction and the vertical supporting line in the same coordinate direction; see Figure \ref{fig1} (left picture). Then $\theta$ equals zero in case the shape is differentiable in the coordinate directions, and $\theta$ is at most $\pi/4$, which occurs if the shape is a diamond. Given $\eps>0$ and $n\in\Z$, we partition the upper half-plane $\Z\times\Z_+$ into two regions $\Hl(n)$ and $\Hr(n)$ as follows: Consider the semi-infinite line through $(n-1/2,0)$ with angle $\theta+\eps$ to the vertical line through the same point (see Figure \ref{fig1}, right picture), and write $\Hl(n)$ for the part of the upper half-plane to the left of this line, excluding points on the line, and $\Hr(n)$ for the part to the right of the line, including points on the line. Finally, define the strips $S_k:=\{(x,y)\in \Z^2:0\leq y\leq k\}$ and $S_k^+=\{(x,y)\in \Z^2: x\geq 0, 0\leq y\leq k\}$.

The proof of the \emph{only if}-direction of Theorem \ref{thm:main} can roughly be divided into three steps, where the first one may be considered the most fundamental:

\begin{description}
\item[Step (i)] Show that, for every $\lambda\geq 1$ and $\eps>0$, if type 2 survives indefinitely, then almost surely type 2 reaches $\Hr(n)$ before type 1 for infinitely many $n\ge1$.
\item[Step (ii)] Show that, for every $\lambda>1$ there exists $\eps>0$ such that, if type 2 comes first to $\Hr(n)$, then for each each $k\ge1$ there is a positive probability (uniform in $n$) that type 2 occupies all vertices in $S_k\cap \Hr(n)$.
\item[Step (iii)] Show that, if type 2 conquers all but finitely many vertices in $S_k^+$ for $k$ large, then it will eventually almost surely defeat type 1.
\end{description}

Combining steps~(i) and~(ii) (or in fact a slight rephrasing of these claims) one obtains that, if type 2 survives indefinitely, then for all $k\geq 1$ it will almost surely conquer all but finitely many sites in the strip $S_k^+$ along the horizontal axis. According to step~(iii), this means that type 1 will eventually become surrounded by type~2, ruling out coexistence.

\begin{figure}[htbp]
\begin{center}
\begin{tikzpicture}[scale=.5]
\def\r{4}
\filldraw[draw=black,fill=gray!40!white] (2,0) -- ({4*cos(45)-2},{4*sin(45)}) -- (-2,4) -- ({4*cos(135)-2},{4*sin(135)}) -- (-6,0) -- cycle;
\draw[->] (-7,0) -- (2.5,0);
\draw[->] (-2,-.5) -- (-2,4.5);
\draw (2,0) -- (2,4);
\draw (2,2) arc [start angle=90, end angle=112.5, radius=2];
\draw (2,2) node[anchor=south east] {$\theta$};
\fill[fill=gray!40!white] (12,0) -- ({12+4*cos(67.5)/sin(67.5)},4) -- (16,4) -- (16,0) -- cycle;
\draw[->] (7,0) -- (16.5,0);
\draw[->] (8,-.5) -- (8,4.5);
\draw[dashed] (12,0) -- (12,4);
\draw (12,2) arc [start angle=90, end angle=67.5, radius=2];
\draw (12,0) -- ({12+4*cos(67.5)/sin(67.5)},4);
\draw (12,0) node[anchor=north] {$(n,0)$};
\draw (13.25,2) node[anchor=south east] {$\scriptscriptstyle{\theta+\varepsilon}$};
\draw (12,1.5) node[anchor=south east] {$\mathcal{L}_\varepsilon(n)$};
\draw (16,1.5) node[anchor=south east] {$\mathcal{R}_\varepsilon(n)$};
\end{tikzpicture}
\end{center}
\caption{Illustration of $\theta$ and the regions $\Hl(n)$ and $\Hr(n)$ in the case that the shape is an octagon. The shape and the region $\Hr(n)$ are shaded.}
\label{fig1}
\end{figure}
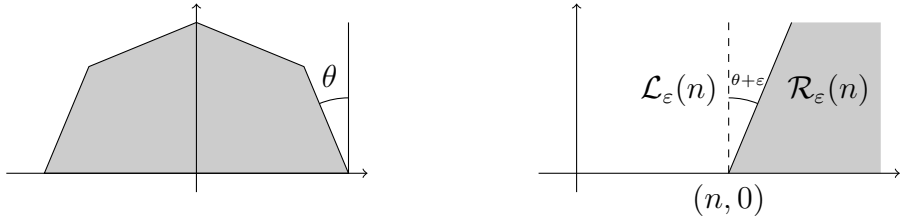

The angle $\theta$ used to define the region $\Hr(n)$ can be motivated as follows: On one hand the claim in step~(ii), which will be a consequence of the shape theorem, cannot hold for any angle larger than $\theta$. On the other hand, while the claim in step~(i) certainly could be correct also for angles smaller than $\theta$ (assuming that $\theta>0$), proving such a thing would require detailed understanding of the structure of infinite one-sided geodesics in the half-plane setting. The information needed would go beyond our current understanding for the analogous objects in the full-plane. Of course, since we believe that the shape is differentiable (at least in coordinate directions) we consequently believe that $\theta=0$, and in this case we cannot do better that having $\Hr(n)$ defined as an $\varepsilon$-tilted vertical line.

The rest of the paper is organized so that relevant background on one-type first passage percolation is given in Section 2. In Section 3 we use Busemann functions to control the evolution of the one-type process to obtain a statement that will establish step (i). Section 4 is devoted to step (ii), which is essentially a consequence of the shape theorem. Finally, the proof of Theorem \ref{thm:main} is completed in Section 5, where step (iii) is established by an adaption of an argument from~\cite{HP_absence}.

\section{Preliminaries}

In standard first passage percolation each edge $e$ of some underlying graph is independently equipped with a non-negative random variable $\tau(e)$ from some common distribution. Throughout this paper, we shall assume that the underlying graph is the upper half-plane $\Z\times\Z_+$, equipped with edges between nearest-neighbors, and that the weights $\{\tau(e)\}$ are unit exponentials. Note that $\{\lambda^{-1}\tau(e)\}$ are then exponentials with parameter $\lambda$. Given a path $\Gamma$, we let $T_\lambda(\Gamma):=\sum_{e\in\Gamma}\lambda^{-1}\tau(e)$ and define the \emph{passage time} between two sets $\Phi,\Psi\subset\Z\times\Z_+$ in the environment $\{\lambda^{-1}\tau(e)\}$ as
$$
T_\lambda(\Phi,\Psi):=\inf\big\{T_\lambda(\Gamma):\Gamma \mbox{ is a path in $\Z\times\Z_+$ connecting $\Phi$ to $\Psi$}\big\},\\
$$
To simplify the notation, we write $T_1(\Gamma)=T(\Gamma)$, $T_1(\Phi,\Psi)=T(\Phi,\Psi)$, and $T_\lambda(x,y)$ for the passage time between $\{x\}$ and $\{y\}$ for $x,y\in\mathbb{Z}^2$. It is immediate from the construction that $T_\lambda(\Phi,\Psi)=\lambda^{-1} T(\Phi,\Psi)$ for all $\lambda\ge1$.

The above construction gives rise to a simultaneous coupling of the two-type processes for all $\lambda\geq 1$, where type~1 requires time $\tau(e)$ to traverse an edge $e$ while type~2 requires time $\lambda^{-1}\tau(e)$. The passage time $T(\mathbf{0},z)$ then denotes the time at which type~1 arrives at the site $z$, unless $z$ is already reached by type 2 by then, and $T_\lambda(\mathbf{1},z)$ similarly denotes the time it would take type~2 to reach $z$, unless impeded by type~1 along the way.\footnote{Throughout the paper, we shall let bold letters like $\nbf$ be short for the horizontal vectors $(n,0)$.} In the case that $\lambda=1$, whether or not a site $z$ is eventually occupied by type~2 can be read out directly from $T$; it will in the case that $T(\mathbf{1},z)<T(\mathbf{0},z)$. Understanding the evolution in the two-type Richardson model thus leads us to recall some basic results for one-type first-passage percolation. Due to the relation $T_\lambda(x,y)=\lambda^{-1}T(x,y)$, we focus in the remainder of this section on the case $\lambda=1$; corresponding results for $\lambda>1$ are obtained by a simple scaling argument. 

Although first passage percolation in half-planes has been studied before, e.g.\ in~\cite{Ahl15,ADH15,WW98}, the vast majority of the literature is concerned with the two and higher dimensional nearest-neighbor lattices. It will be convenient to survey some of the results here. In analogy with the notation in the half-plane, we shall denote by $\overline T(\Phi,\Psi)$ the passage time between the two sets $\Phi,\Psi\subset\Z^2$, where the infimum is now taken over paths in $\Z^2$ connecting $\Phi$ and $\Psi$.

A first crucial observation is that $\overline T$ defines a metric on $\Z^2$. In particular, it is subadditive in the sense that
$$
\overline T(x,y)\leq \overline T(x,z)+\overline T(z,y)\quad\mbox{for all }x,y,z\in\Z^2.
$$
Using subadditive ergodic theory \cite{K68,L85}, one can establish the existence of a time constant $\mu\in(0,\infty)$ specifying the asymptotic inverse speed of the growth along the axes. Specifically, we have that
$$
\lim_{n\to\infty}\frac{\overline T(\org,\nbf)}{n}=\mu\quad \mbox{almost surely and in }L^1.
$$
This can be extended to an arbitrary direction in the first octant, and hence by symmetry of $\Z^2$, to any arbitrary direction: For $\alpha\in[0,\pi/4]$, let $u_\alpha$ denote a unit vector with angle $\alpha$ to the $x$-axis, that is, $u_\alpha=(\cos\alpha,\sin\alpha)$. Also, for $x,y\in\R^2$, define $\overline T(x,y):=\overline T(x',y')$, where $x'$ and $y'$ are the points in $\Z^2$ closest to $x$ and $y$, respectively. Then there exists a directional time constant $\mu_\alpha\in(0,\infty)$ such that
\begin{equation}\label{dir_mu}
\lim_{n\to\infty}\frac{\overline T(\org,nu_\alpha)}{n}=\mu_{\alpha}\quad \mbox{almost surely and in }L^1.
\end{equation}
By definition, we have $\mu_0=\mu$. We remark that passage times to lines rather than single points obey the same asymptotics. For instance, with $\bar{\ell}_\alpha(n)$ denoting the straight line through $nu_\alpha$ with normal vector $u_\alpha$, we have that $\frac1n\overline T(\org,\bar{\ell}_\alpha(n))$ converges to $\mu_\alpha$ almost surely. This can be seen as a consequence of the fundamental shape theorem, which in its first version dates back to the work of Richardson~\cite{R73}.

Since $\overline T$ defines a (random) metric on $\Z^2$ it is natural to investigate the shape of large balls in this metric. The shape theorem~\cite{Kes73,R73} states that the set of sites that can be reached from the origin within time $t$ converges almost surely on the scale $t^{-1}$ to a deterministic shape $A$, that is, with probability one, we have for every $\eps>0$ that $W(t):=\{x\in\R^2:\overline T(\org,x)\leq t\}$ satisfies
$$
(1-\eps)A\subset \frac{W(t)}{t}\subset (1+\eps)A \quad\mbox{for all large }t.
$$
The asymptotic shape $A$ can be characterized as the unit ball in the norm defined by $\mu(x)=\lim_{n\to\infty}\frac1n\overline T(\org,nx)$ for $x\in\R^2$. It is thus known to be compact and convex, with non-empty interior, and it inherits all symmetries of $\Z^2$. Apart from this, very little is known about the properties of the shape. It has been studied by aid of simulations in \cite{AlmDeijfen15}, where the results indicate that it is close to, but not identical to, a Euclidean disk. We remark that there is no theoretical support for $A$ being a Euclidean disk, and in large dimension it is known not to be a Euclidean ball.

When restricting the growth to a strip $S_k:=\{(x,y)\in \Z^2:0\leq y\leq k\}$ for some $k\ge1$, the speed of progression decreases. However, the thicker the strip, the smaller is the effect. To be precise, let $\Tk(\Phi,\Psi)$ denote the passage time between $\Phi\subset S_k$ and $\Psi\subset S_k$, where the infimum is taken over paths $\Gamma\subset S_k$ connecting $\Phi$ and $\Psi$. Again, the subadditive ergodic theorem shows that $\frac1n\Tk(\org,\nbf)$ converges (almost surely and in $L^1$) to some constant $\muk\in(0,\infty)$. Moreover,
\begin{equation}\label{mu_strip}
\muk\searrow\mu\quad\mbox{as }k\to\infty;
\end{equation}
see e.g.~\cite[Proposition 8]{Ahl15}. A similar statement holds for directions other than the axes directions. As a consequence, one can show that a shape theorem holds also for first passage percolation in the upper half-plane $\Z\times\Z_+$, and that the asymptotic shape in this case is the half-plane restriction of the shape $A$ arising in the full-plane growth; see \cite[Theorem 1]{Ahl15}. We shall occasionally need the following stronger form of this half-plane shape theorem, which is a consequence e.g.\ of~\cite[Proposition~15]{Ahl15}: For every $\eps>0$ we have, almost surely, for all $y$ and all but finitely many $z$ in $\Z\times\Z_+$ that
\begin{equation}\label{eq:better_shape}
\big|T(y,z)-\mu(z-y)\big|<\eps\max\{|z|,|z-y|\},
\end{equation}
where $\mu$ is the time constant as determined by $\overline T$.

\section{A one-type lemma}

The aim of this section is to take the first and most fundamental step towards a proof of our main theorem. It will be crucial for ruling out coexistence in the case when $\lambda>1$, but we will use it also to give a short proof of coexistence in the case when $\lambda=1$. The result is a statement for the one-type process on $\Z\times\Z_+$.

\begin{lma}\label{lma:main}
For every $\eps>0$ there exists $\gamma>0$ such that
$$
\Pr\Big(T(-\nbf,\org)<T(-\nbf,\Hr(0)\backslash \{\org\})\text{ for all }n\ge1\Big)>\gamma.
$$
\end{lma}

Key to the proof of the lemma will be the notion of Busemann functions. Define, for all $n\ge1$ and sites $u,v$ in the half-plane, the Busemann-like function
$$
B_n(u,v):=T(-\nbf,u)-T(-\nbf,v).
$$
Lemma~\ref{lma:main} can be rephrased to say that with positive probability $B_n(\org,v)<0$ for all $v\in\Hr(0)\setminus\{\org\}$ and $n\ge1$. We shall first show that, almost surely, $B_n(\org,v)<0$ may fail for some $n$ for at most finitely many $v$ (Lemma \ref{lma:negative}). A local modification argument will then show that with positive probability it does not.

A key observation is that, for fixed $m\ge1$, the sequence $\{B_n(\org,\mbf)\}_{n\ge1}$ is almost surely increasing. The limit
$$
B(\org,\mbf):=\lim_{n\to\infty}B_n(\org,\mbf)
$$
hence exists almost surely. Indeed, this turns out to be true for all $u$ and $v$, see~\cite{ADH15}, but we shall not need this fact. Instead, we shall make use of the following asymptotic property.

\begin{lma}\label{lma:Busemann}
For all $m\ge1$, we have that $\E[B(\org,\mbf)]=-\mu\cdot m$, and
$$
\lim_{m\to\infty}\frac1mB(\org,{\bf m})=-\mu\quad\text{almost surely}.
$$
\end{lma}

\begin{proof}
A useful property of $B_n$ is that it is additive. The additivity carries over in the limit as $n\to\infty$ and for $B$ this implies that
\begin{equation}\label{eq:B}
\frac1mB(\org,\mbf)=\frac1m\sum_{j=0}^{m-1}B({\bf j},{\bf j+1}),
\end{equation}
where $B({\bf j},{\bf j+1}):=\lim_{n\to\infty}B_n({\bf j},{\bf j+1})$. Due to invariance with respect to horizontal shifts, sending $m$ to infinity in~\eqref{eq:B}, the ergodic theorem yields the almost sure limit $\E[B(\org,\ett)]$. By additivity, it only remains to identify $\E[B(\org,\ett)]$ with $-\mu$.

To this end, we rephrase $B(\org,\ett)$ as a limit of partial averages, and obtain
$$
\E[B(\org,\ett)] = \E\Big[\lim_{n\to\infty}\frac1n\sum_{j=0}^{n-1}B_j(\org,\ett)\Big]=\lim_{n\to\infty}\frac1n\sum_{j=0}^{n-1}\E[B_j(\org,\ett)],
$$
where extraction of the limit is allowed by dominated convergence, since $|B_j(\org,\ett)|\le T(\org,\ett)$. Due to invariance with respect to horizontal shifts, we have further that
$$
\E[B(\org,\ett)]=\lim_{n\to\infty}\frac1n\sum_{j=0}^{n-1}\E[T(\org,{\bf j})-T(\org,{\bf j+1})]=\lim_{n\to\infty}\frac1n\E[-T(\org,{\bf n})]=-\mu,
$$
as required.
\end{proof}

Let $\partial \Hr(n)$ denote the set of sites in $\Hr(n)$ that have at least one neighbor in $\Hl(n)$.

\begin{lma}\label{lma:negative}
There exists $\delta>0$ such that, with probability one, for all $n\ge1$ and all but finitely many $v\in\partial\Hr(0)$, we have that
$$
B_n(\org,v)<-\delta|v|\mu<0.
$$
\end{lma}

\begin{proof}
Note that, by convexity of the shape and the definition of $\theta$, there exists $\delta>0$ such that for each $v\in\partial \Hr(0)$ there is $m=m(v)$ such that
$$
\mu(v-\mbf)\le(1-\delta)\mu(\mbf);
$$
see Figure \ref{fig2} (left picture).
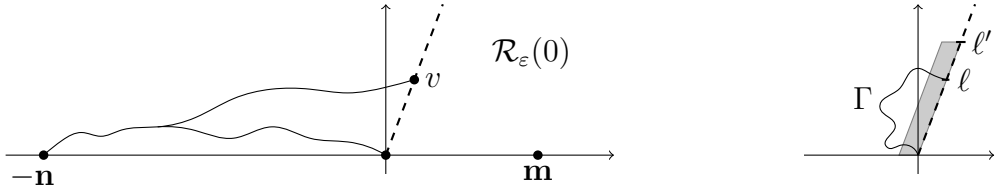
\begin{figure}[htbp]
\begin{center}
\begin{tikzpicture}[scale=.5]
\draw[->] (-10,0) -- (6,0);
\draw[->] (0,-.5) -- (0,4);
\draw[dashed,thick] (0,0) -- (1.5,4);
\filldraw[fill=gray!40!white,draw=gray] (14,0) -- (13.5,0) -- (14.625,3) -- (15.125,3) -- cycle;
\draw[->] (11,0) -- (16,0);
\draw[->] (14,-.5) -- (14,4);
\draw[dashed,thick] (14,0) -- (15.5,4);
\draw (-9,0) to [curve through={(-8.5,.4) (-8,.6) (-7.5,.5) (-7,.65) (-6,.75) (-5,1) (-4,1.5) (-3,1.75) (-2,1.75) (-1,1.65) (0,1.8)}] (.75,2);
\draw (-6,.75) to [curve through={(-5,.65) (-4,.45) (-3,.75) (-2,.45) (-1,.35)}] (0,0);
\draw (14,0) to [curve through={(13.5,.25) (13.2,.2) (13.4,.75) (13,1.25) (13.5,1.75) (14,2.3) (14.25,2.2)}] (14.75,2);
\draw (-9.3,0) node[anchor=north] {$-\mathbf{n}$};
\draw (4,0) node[anchor=north] {$\mathbf{m}$};
\draw (12,1.5) node[anchor=west] {$\Gamma$};
\draw (.75,2) node[anchor=west] {$v$};
\draw (2.5,2.7) node[anchor=west] {$\Hr(0)$};
\draw (14.75,2) node[anchor=west] {$\ell$};
\draw (15.125,3) node[anchor=west] {$\ell'$};
\fill[fill=black] (0,0) circle (.125);
\fill[fill=black] (.75,2) circle (.125);
\fill[fill=black] (-9,0) circle (.125);
\fill[fill=black] (4,0) circle (.125);
\draw[thick] (14.625,2) -- (14.825,2);
\draw[thick] (15,3) -- (15.25,3);
\end{tikzpicture}
\end{center}
\caption{Geometry in the proofs of Lemmas~\ref{lma:negative} and~\ref{lma:main}.}
\label{fig2}
\end{figure}
Indeed, $m$ can be chosen to roughly equal $c|v|$ for some $c>0$. Together with the strong version of the shape theorem stated in~\eqref{eq:better_shape}, it follows that almost surely for all but finitely many $v\in\partial \Hr(0)$ we have that
$$
B_n(\mbf,v)\le T(\mbf,v)\le(1+\delta)\mu(v-\mbf)\le(1-\delta^2)\mu(\mbf).
$$
Moreover, by monotonicity of $B_n$ and Lemma~\ref{lma:Busemann}, we have almost surely for all $n\ge1$ and large $m$ that
$$
B_n(\org,\mbf)\le B(\org,\mbf)\le-(1-\delta^2/2)\mu(\mbf).
$$
Combining the two estimates we conclude that almost surely, for all but finitely many $v\in \partial \Hr(0)$, we have for all $n\ge1$ that
\begin{equation*}
B_n(\org,v)=B_n(\org,\mbf)+B_n(\mbf,v)\le-(\delta^2/2)\mu(\mbf)<0.
\end{equation*}
Since $m$ is roughly $c|v|$ for some $c>0$, the lemma follows.
\end{proof}

\begin{proof}[Proof of Lemma~\ref{lma:main}]
Let $(x_n,y_n)$ be the point in $\Hr(0)$ with the smallest passage time to $-\nbf$. By Lemma~\ref{lma:negative}, the sequence $(y_n)_{n\ge1}$ is almost surely bounded. Fix $\ell$ large so that, with probability at least $3/4$, we have $y_n\leq \ell$ for all $n$. Then pick some finite path $\Gamma$, connecting the origin to a point in $\partial\Hr(0)$ of the form $(x,\ell+1)$, which except for its endpoints is contained in $\Hl(0)$; see Figure \ref{fig2} (right picture). Next, take $t$ large so that, with probability at least $3/4$, the total passage time $T(\Gamma)$ is at most $t$. Note that, since any path from $-\nbf$ to $(x,y)\in\partial\Hr(0)$ with $1\leq y\leq \ell$ must hit $\Gamma$ before hitting $(x,y)$, we have for all $n\geq 1$ that, on the intersection of the above two events,
$$
T(-\nbf,\org)\leq T(-\nbf,\Hr(0))+T(\Gamma)\le T(-\nbf,\Hr(0))+t.
$$
Write $U_{\ell'}$ for the set of sites $(x,y)\in \partial\Hr(0)$ with $y\geq \ell'$. Due to Lemma~\ref{lma:negative}, we may pick $\ell'\geq \ell$ such that $T(-\nbf,U_{\ell'})\geq T(-\nbf,\org)+2t$ for all $n\ge1$ with probability at least $3/4$. Define $C$ to be the intersection of all three events above. That is, let
$$
C:=\big\{y_n\leq \ell\text{ for all }n\big\}\cap \big\{T(\Gamma)\leq t\big\}\cap \big\{T(-\nbf,U_{\ell'})\geq T(-\nbf,\org)+2t\text{ for all }n\big\},
$$
and note that $\Pr(C)\ge1/4$.

Let $\Lambda_{\ell'}$ denote the set of edges connecting sites $(x,y)\in\partial\Hr(0)\setminus\{\org\}$ with $y\leq \ell'$ to sites in $\Hl(0)$; see Figure \ref{fig2} (shaded area in the right picture). We complete the proof by arguing that, on the event $C$, a configuration where the origin is the closest point in $\partial\Hr(0)$ to $-\nbf$ for all $n\geq 1$ is obtained by increasing the weight of all edges in $\Lambda_{\ell'}$ to $2t$. Indeed, the time minimizing path from $-\nbf$ to $\Hr(0)$ will then not hit a point $(x,y)\in\Hr(0)$ for $y=1,\ldots,\ell'$, since it would have reached the origin via $\Gamma$ before the last edge is traversed. It will also not hit $\Hr(0)$ for $y\geq \ell'$, since it will take at least time $2t$ from the moment when $\Gamma$ is hit to reach that level.

To formalize this, we define another i.i.d.\ family of edge weights $\{\hat\tau(e)\}$, where $\hat\tau(e)=\tau(e)$ for $e\not\in\Lambda_{\ell'}$ and where $\hat\tau(e)$ is sampled independently of $\tau(e)$ for $e\in\Lambda_{\ell'}$. Denoting by $Q$ the event $\{\hat{\tau}(e)>t\mbox{ for all }e\in\Lambda_{\ell'}\}$, and distances with respect to $\{\hat\tau(e)\}$ by $\hat T$, the above reasoning gives that
$$
\Pr\left(\hat{T}(-\nbf,\org)<\hat{T}(-\nbf,\Hr(0)\backslash \{\org\})\text{ for all }n\ge1\right)
\geq \Pr\big(C\cap Q\big)=\Pr(C)\Pr(Q)>0,
$$
due to independence of the two configurations on $\Lambda_{\ell'}$. Since the two configurations are equal in distribution, the lemma follows.
\end{proof}

\section{A two-type lemma}

The next lemma concerns the two-type process with an unbounded initial configuration. It applies when type 2 is strictly stronger than type 1, and is derived as a geometric consequence of the shape theorem. Recall that $S_k^+=\{(x,y)\in \Z^2: x\geq 0, 0\leq y\leq k\}$. Note also that for small enough values of $\eps>0$ the origin is the only site on the horizontal axis contained in $\partial\Hr(0)$.

\begin{lma}\label{lma:chance}
For every $\lambda>1$ there is $\eps>0$ such that if initially $\org$ is occupied by type 2 and all sites in $\partial\Hr(0)\backslash\{\org\}$ are occupied by type 1, then, for every $k\ge1$, there is a positive probability that type 2 occupies all initially uninfected sites in the half-strip $S_k^+$.
\end{lma}

\begin{proof}
Fix $\lambda>1$. Note that it suffices to prove the lemma for large $k$, since if type 2 occupies all uninfected sites in $S_k^+$, then this is trivially the case also for all $k'\leq k$. By~\eqref{mu_strip} we thus pick $k$ large so that
$$
\lambda^{-1}\mu\le\lambda^{-1}\muk<\mu.
$$
Let $\delta=(\mu-\lambda^{-1}\muk)/4$ and set $\rho=\lambda^{-1}\muk+2\delta_0$.

It follows from the half-plane shape theorem (the version stated in~\eqref{eq:better_shape}), convexity of the shape and the definition of $\Hr(0)$ that, almost surely, for large $n$ we have that
$$
T(\partial\Hr(0),(n,k))>(\mu-\delta')n
$$
for some $\delta'=\delta'(\eps)>0$, with $\delta'\to 0$ as $\eps\to 0$. Hence, for $\eps>0$ small, $\mu-\delta'\geq\rho+\delta$. Moreover, almost surely, we have that
$$
\Tk_\lambda(\org,(n,k))<(\rho-\delta)n
$$
for all large $n$. Finally, write $\tilde{T}^{\sss (k)}_\lambda(\org,(n,k))$ for the above passage time in the process based on  $\{\lambda^{-1}\tau(e)\}$, when sites in $\partial\Hr(0)$ cannot be used, and note that this clearly obeys the same asymptotics. (We assume here and in what follows that $\eps>0$ is small, so that the origin is the only site on the horizontal axis contained in $\partial\Hr(0)$.)

For $m\ge1$ now define
\begin{equation*}
\begin{aligned}
D_m&:=\{T(\partial\Hr(0),(n,k))>(\rho+\delta)n\mbox{ for all }n\geq m\},\\
D'_m&:=\{\tilde{T}^{\sss (k)}_\lambda(\org,(n,k))<(\rho-\delta)n\mbox{ for all }n\geq m\},
\end{aligned}
\end{equation*}
and pick $m$ large so that $\Pr(D_m\cap D_m')>3/4$. Let $\Omega_m$ denote the set of edges consisting of all edges connecting an initially type 1 infected site to a neighbor in $S_k^+$, and all vertical edges connecting a site $(j,k+1)$ in $\Hr(0)$ with $j\le m$ to $(j,k)$. Hence $\Omega_m$ consists of all edges up to the level $x=m$ through which type 1 can enter the strip; see Figure~\ref{fig3}.
\begin{figure}[htbp]
\begin{center}
\begin{tikzpicture}[scale=.75]
\draw[->] (-1,0) -- (9,0);
\draw [->] (0,-.5) -- (0,4);
\draw (0,2) node[anchor=east] {$k$};
\draw (7,0) node[anchor=north] {$m$};
\filldraw[fill=gray!40!white,draw=gray] (0,0) -- (0.86,2.3) -- (7,2.3) --(7,2) -- (1.1,2) -- (0.3,0) -- cycle;
\draw[dashed] (0,0) -- (1.5,4);
\draw[dashed] (0,2) -- (9,2);
\draw[dashed] (7,0) -- (7,4);
\draw (3,1) node[anchor=west] {$\Omega'_m$};
\draw (4,3) node[anchor=west] {$\Omega_m$};
\draw[ultra thin, ->] (4,3) -- (2.8,2.15);
\end{tikzpicture}
\end{center}
\caption{The regions $\Omega_m$ (shaded) and $\Omega_m'$.}
\label{fig3}
\end{figure}
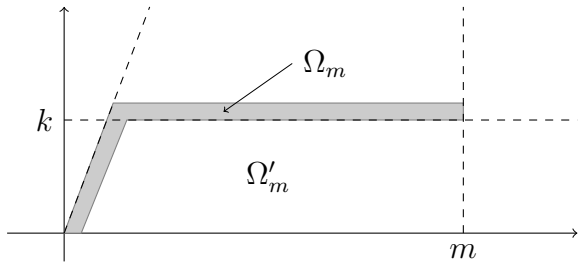
Also, let $\Omega_m'$ denote the set of edges connecting initially uninfected sites in $S_k^+$ up to level $x=m$, and note that $\Omega_m$ and $\Omega_m'$ are disjoint.

Next, let
\begin{equation*}
\begin{aligned}
E_{m,t}&:=\{\mbox{$\tau(e)>tkm$ for all $e\in\Omega_m$}\},\\
E_{m,t}'&:=\{\mbox{$\lambda^{-1}\tau(e)<t$ for all $e\in\Omega_m'$}\}.
\end{aligned}
\end{equation*}
Since $\Pr(E_{m,t}')\to 1$ as $t\to\infty$, we can pick $t$ large so that $\Pr(D_m\cap D_m'\cap E_{m,t}')>1/2$.
We claim that, on $D_m\cap D_m'\cap E_{m,t}\cap E_{m,t}'$, type 2 occupies all initially uninfected sites in $S_k^+$. To see this, note that $E_{m,t}\cap E_{m,t}'$ ensures that type 1 cannot enter the strip at a site $(j,k)$ with $j<m$, since any such site can be reached from the origin by a path in $\Omega_m'$ with weight at most $mkt$. The event $D_m\cap D_m'$ then guarantees that type 1 cannot enter the strip at a site $(j,k)$ with $j\geq m$, since type 2 is faster to all such sites once it has access to the initial piece of the strip.

It remains to prove that $\Pr(D_m\cap D_m'\cap E_{m,t}\cap E_{m,t}')>0$. To this end, write
$$
\Pr(D_m\cap D_m'\cap E_{m,t}\cap E_{m,t}')=\Pr(D_m\cap D_m'\cap E_{m,t}'|E_{m,t})\Pr(E_{m,t}).
$$
The events $D_m'$ and $E_{m,t}'$ involve only edges in $\Omega_m'$ while $E_{m,t}$ involves only edges in $\Omega_m$. Hence, since $\Omega_m'$ and $\Omega_m$ are disjoint, the conditioning on $E_{m,t}$ does not affect $D_m'$ and $E_{m,t}'$. As for $D_m$, the event $E_{m,t}$ stipulates that the passage times on certain edges are large. This clearly increases the probability of $D_m$ so that, in summary, $\Pr(D_m\cap D_m'\cap E_{m,t}'|E_{m,t})\geq \Pr(D_m\cap D_m'\cap E_{m,t}')$. The desired conclusion follows by noting that $\Pr(E_{m,t})>0$ since $\Omega_m$ is finite and $t$ fixed.
\end{proof}

\section{Proof of Theorem \ref{thm:main}}

In this section we prove Theorem \ref{thm:main}. As mentioned, there are a number of proofs in the literature showing that coexistence is possible on $\Z^2$ when $\lambda=1$, and some of these are easily adapted to show the same statement in the half-plane. However, this can also be obtained by a short argument based on Lemma~\ref{lma:main}.

\begin{proof}[Proof of the {\rm{if}}-direction of Theorem \ref{thm:main}]
Take $\lambda=1$. Let $F$ denote the event in Lemma~\ref{lma:main}, and let $\bar F$ denote its reflection in the vertical axis. Let further $\bar F_m$ denote the translate of $\bar F$ along the vector $\mbf$. We observe that, on $F$, type 1 will be first to all sites along the negative horizontal axis. Similarly, on $\bar F_1$, type 2 will be first to all sites along the positive horizontal axis. Although there is no guarantee that the intersection of the two events occurs with positive probability, since $\bar F_m$ occurs with a density (due to the ergodic theorem), we may fix $m\ge1$ so that $\Pr(F\cap\bar F_m)>0$. To guarantee coexistence, it then remains to show that, on $F\cap\bar F_m$, there is positive probability for type 2 to reach $(m,0)$ before type 1 reaches $\Hr(m)$.

Let $O$ denote the event that each edge adjacent to the origin has weight at least $\delta$, and note that $\Pr(F\cap\bar F_m\cap O)>0$ for small $\delta>0$. Let $O'$ denote the event that the sum of the weights on the edges along the axis connecting $\ett$ to $\mbf$ is at most $\delta/2$. Note that, on $O\cap O'$, type 2 will reach $\mbf$ before type 1 takes its first step. Since $F$, $\bar F_m$ and $O$ are independent of the state of the edges defining $O'$, it follows that
$$
\Pr(\coex)\geq\Pr(F\cap\bar F_m\cap O\cap O')=\Pr(F\cap\bar F_m\cap O)\Pr(O')>0,
$$
as required.
\end{proof}

We proceed with the \emph{only if}-direction, and start by combining Lemmas~\ref{lma:main} and~\ref{lma:chance} into a statement for the two-type process.

\begin{lma}\label{lma:strip}
For every $\lambda>1$ and $k\geq 1$, if type 2 occupies infinitely many sites in the two-type model on $\Z\times \Z^+$, then type 2 will almost surely occupy all but finitely many vertices in $S_k^+$.
\end{lma}

\begin{proof}
Fix $\lambda>1$ and $k\ge1$. Write $F$ for the event in Lemma~\ref{lma:main}, and let $F_m$ denote the translate of $F$ along the vector $\mbf$. Also, let $G$ denote the event in Lemma~\ref{lma:chance}, and let $G_m$ denote the translate of $G$ along the vector $(m,0)$. Each of the two events $F$ and $G$ occur with positive probability. Moreover, $F$ is determined by edges between sites in $\Hl(0)\cup \partial\Hr(0)$ involving at least one site in $\Hl(0)$, while $G$ is determined by edges between pairs of sites in $\Hr(0)$. Hence, the two events are independent and $\Pr(F\cap G)>0$. By the ergodic theorem, $F_m\cap G_m$ will occur for infinitely many $m\ge1$, almost surely.

It remains to prove that, on the event $F_m\cap G_m\cap \{\text{type 2 survives}\}$, where $m\ge1$, type 2 occupies all but finitely many vertices in $S_k^+$. For this, clearly it suffices to see that, if type 2 survives indefinitely, then $F_m$ implies that type 2 reaches $\mbf$ before any other site in $\partial\Hr(m)$ is reached by type 1. To this end, let $\Gamma$ denote the time minimizing path from the origin to $\mbf$. Note that, if type 2 survives indefinitely, then $\mbf$ must be occupied by type 2 in the two-type process. Let $v$ denote the first (in time) point on the path $\Gamma$ that is occupied by type 2 in the two-type process. The fastest way to get from $v$ to $\Hr(m)$ is to follow $\Gamma$ and, doing this, type 2 will arrive at $\mbf$ before any other site in $\Hr(m)$ is infected, as desired.
\end{proof}

The last ingredient we need in order to prove the \emph{only if}-direction of Theorem~\ref{thm:main} is a half-plane version of a result from~\cite[Proposition 2.2]{HP_absence}. More precisely, we need to show that, if type 2 conquers a wide half-strip, then type 2 will end up surrounding type 1. The argument will be similar to that of~\cite{HP_absence}, but the geometric construction is easier in our case and the proof consists of applying the ideas in Lemmas~\ref{lma:chance} and~\ref{lma:strip} in non-axis directions. We shall therefore be brief.

\begin{lma}\label{lma:speed}
For every $\lambda>1$, there is $k\ge1$ such that, if type 2 occupies all but finitely many sites in $S_k^+$, then almost surely type 1 will occupy only finitely many sites.
\end{lma}

\begin{proof}
If type 2 occupies all but finitely many sites in the half-strip $S_k^+$ for $k$ sufficiently large, then the type 2 speed along the axis in $S_k^+$ will be strictly larger than the speed of type 1. As we shall see, type 2 will then be strictly faster than type 1 also in direction $\alpha$, for some small $\alpha>0$. This can be used to show that type 2 occupies all but finitely many vertices in an $\alpha$-cone. By repeating the argument we then show that type~2 will also occupy almost all sites in a $2\alpha$-cone, etc.

Recall the definition, in~\eqref{dir_mu}, of the time constant $\mu_\alpha$ in direction $\alpha$ based on unit rate exponential edge weights. The time constant in direction $\alpha$ based on exponential edge weights with parameter $\lambda$ is then given by $\lambda^{-1}\mu_\alpha$. As is well-known, the directional time constant $\mu_\alpha$ is Lipschitz continuous, since $\mu$ defines a norm. In particular, there exists a constant $c>0$ such that, for any $\alpha_0,\alpha\in[0,2\pi]$, we have that 
$$
\mu_{\alpha_0+\alpha}\leq \mu_{\alpha_0}(1+c\alpha).
$$
It follows that, uniformly in the choice of $\alpha_0$, we have $\lambda^{-1}\mu_{\alpha_0+\alpha}\leq \mu_{\alpha_0}$ if $\alpha$ is sufficiently small. By picking $\alpha$ even smaller, we further obtain that $\lambda^{-1}\mu_{\alpha_0+\alpha}<\mu_{\alpha_0}\cos\alpha$. For the remainder of this proof we fix $\alpha>0$ so that for all $\alpha_0\in[0,2\pi]$ we have
\begin{equation}\label{eq:alpha_choice}
\lambda^{-1}\mu_{\alpha_0+2\alpha}<\mu_{\alpha_0}\cos(2\alpha).
\end{equation}

Let $\ell_\alpha(0)$ denote the semi-infinite line starting at the origin with angle $\alpha$ to the horizontal axis. In a first step, we argue that if type~2 occupies all sites in a thick strip, then type 2 will almost surely occupy all but finitely many sites below the line $\ell_\alpha(0)$. Pick $k$ large so that $\lambda^{-1}\muk<\mu$, which is possible by \eqref{mu_strip}. Let $H_m$ be the event that type~2 eventually occupies the site $(m,k)$, and that at the time at which this occurs type~1 has not yet reached the vertical line $L(m)=\{(x,y)\in\mathbb{Z}^2:x=m\}$. The choice of $k$ assures that, given that type~2 captures all but finitely many sites in the strip $S_k^+$, the probability of $H_m$ tends to one as $m$ tends to infinity.

Write $\ell_{2\alpha}(m,k)$ for the semi-infinite line starting at the point $(m,k)$ with angel $2\alpha$ to the horizontal line through $(m,k)$. By~\eqref{eq:alpha_choice}, we have that $\lambda^{-1}\mu_{2\alpha}< \mu\cos(2\alpha)$, and hence that the asymptotic type 2 time from $(m,k)$ to a point on $\ell_{2\alpha}(m,k)$ far from $(m,k)$ is strictly smaller than the type 1 passage time from $L(m)$ to the same point; see Figure \ref{fig4}.
\begin{figure}[htbp]
\begin{center}
\begin{tikzpicture}[scale=.5]
\draw[->] (-5,-1) -- (9,-1);
\draw[->] (-4,-2) -- (-4,5);
\draw[dashed] (0,-1) -- (0,5);
\draw[dashed] (-4,0) -- (9,0);
\draw (-5,0) node[anchor=west] {$k$};
\draw (0,-1) node[anchor=north] {$m$};
\fill[fill=black] (0,0) circle (.125);
\fill[fill=black] (5.6,3) circle (.125);
\draw[thick] (0,0) -- (7.5,4);
\draw (7.5,4) node[anchor=north west] {$\ell_{2\alpha}(m,k)$};
\draw (1,0) arc [start angle=0, end angle=27, radius=1];
\draw (1.1,0.4) node[anchor=west] {$2\alpha$};
\draw[dashed] (0,3) -- (5.6,3);
\draw (2.2,1.8) node[anchor=west] {$x$};
\draw (1,3.3) node[anchor=west] {$x\cos(2\alpha)$};
\end{tikzpicture}
\end{center}
\caption{The line $\ell_{2\alpha}(m,k)$.}
\label{fig4}
\end{figure}
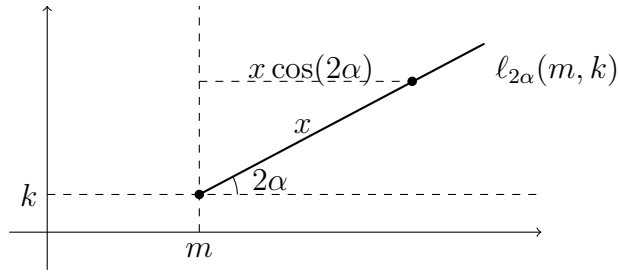
(Here, we say that a point $z\in\R^2$ is infected when the closest point in $\Z^2$ is infected.) Let $G_m^{\sss2\alpha}$ denote the event that, starting from a configuration in which $(m,k)$ is of type 2 and the rest of the line $L(m)$ is of type 1, every point along the line $\ell_{2\alpha}(m,k)$ is eventually captured by type 2. A similar argument as that used to prove Lemma~\ref{lma:chance} then shows that $G_m^{\sss2\alpha}$ occurs with positive probability.
The ergodic theorem implies that $G_m^{\sss2\alpha}$ occur for a positive density of all $m\ge1$, almost surely, and since the conditional probability that $H_m$ occurs, given that type~2 takes the strip, tends to one, their intersection will occur for some (large) value of $m$ almost surely. The occurrence of $H_m\cap G_m^{\sss 2\alpha}$ guarantees that type 2 captures the whole line $\ell_{2\alpha}(m,k)$,
and consequently that the whole area below the line $\ell_{2\alpha}(m,k)$ is captured by type 2. Since $\ell_\alpha(0)$ eventually enters this region, we conclude that if type~2 captures all but finitely many sites in $S_k^+$ (and $k$ is large), then almost surely type~2 captures all but finitely many sites in the cone below the line $\ell_\alpha(0)$.

In a second step we show that for any $\alpha_0>0$, if type 2 occupies all but finitely many vertices in the $\alpha_0$-cone below the line $\ell_{\alpha_0}(0)$, then the same is true for the $(\alpha_0+\alpha)$-cone below the line $\ell_{\alpha_0+\alpha}(0)$, almost surely. Since $\alpha_0$ is arbitrary, this will complete the proof of the lemma. We repeat the argument above, and let $v_m$ denote the point on $\ell_{\alpha_0}(0)$ at distance $m$ from the origin, write $\ell_{2\alpha}^{\alpha_0}(m)$ for the semi-infinite line starting at $v_m$ with angel $2\alpha$ to $\ell_{\alpha_0}(0)$, and let $\bar{\ell}_{\alpha_0}(m)$ be the line through $v_m$ that is orthogonal to $\ell_{\alpha_0}(0)$; see Figure \ref{fig5}.
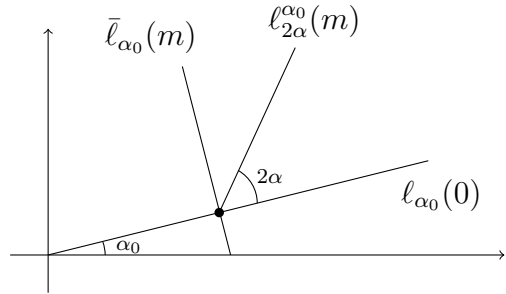
\begin{figure}[htbp]
\begin{center}
\begin{tikzpicture}[scale=.5]
\draw[->] (-1,0) -- (12,0);
\draw[->] (0,-1) -- (0,6);
\draw (0,0) -- (10,2.5);
\draw (9,1.6) node[anchor=west] {$\ell_{\alpha_0}(0)$};
\draw (4.8,0) -- (3.53,5);
\draw (4.2,5.8) node[anchor=east] {$\bar{\ell}_{\alpha_0}(m)$};
\draw (1.5,0) arc [start angle=0, end angle=20, radius=1];
\draw (1.5,0.22) node[anchor=west] {$\scriptstyle\alpha_0$};
\draw (4.5,1.125) -- (6.5,5.5);
\draw (5.5,6.2) node[anchor=west] {$\ell_{2\alpha}^{\alpha_0}(m)$};
\draw (5.5,1.375) arc [start angle=0, end angle=60, radius=1];
\draw (5.2,2.1) node[anchor=west] {$\scriptstyle2\alpha$};
\fill[fill=black] (4.5,1.125) circle (.125);
\end{tikzpicture}
\end{center}
\caption{Lines through the point $v_m$.}
\label{fig5}
\end{figure}
Now, if type 2 occupies all but finitely many vertices in the $\alpha_0$-cone, then its asymptotic speed in direction $\alpha_0$ is determined by $\lambda^{-1}\mu_{\alpha_0}$. Hence, if type~2 occupies all but finitely many vertices in the $\alpha_0$-cone, then the event $H_m^{\alpha_0}$ that type 1 has not yet reached $\bar{\ell}_{\alpha_0}(m)$ when type 2 reaches $v_m$ has probability tending to one as $m\to\infty$.

Furthermore, by~\eqref{eq:alpha_choice} it again follows that the type~2 time
from $v_m$ to a point far along the line $\ell_{2\alpha}^{\alpha_0}(m)$ is with high probability strictly smaller than the type~1 passage time from $\bar{\ell}_{\alpha_0}(m)$ to the same point. Again repeating the argument in the proof of Lemma~\ref{lma:chance}, we may show that the event $G_m^{\sss\alpha_0,2\alpha}$ that the whole line $\ell_{2\alpha}^{\alpha_0}(m)$ is captured by type 2, when starting from a configuration where $v_m$ is of type 2 and the rest of the sites on or to the left of the line $\bar{\ell}_{\alpha_0}(m)$ is of type 1, occurs with positive probability. Appealing to the ergodic theorem we again find that, given that type~2 takes all but finitely many sites in the $\alpha_0$-cone, the event $H_m^{\alpha_0}\cap G_m^{\alpha_0,2\alpha}$ will occur for some (large) $m$, almost surely, and so type~2 will occupy all but finitely many sites in the $(\alpha_0+\alpha)$-cone below the line $\ell_{\alpha_0+\alpha}(0)$. Since $\alpha_0$ was arbitrary, this completes the proof.
\end{proof}

\begin{proof}[Proof of {\rm{only if}}-direction of Theorem \ref{thm:main}]
The \emph{only if}-direction of Theorem~\ref{thm:main} is an immediate consequence of Lemmas~\ref{lma:strip} and~\ref{lma:speed}.
\end{proof}

\end{document}